\documentclass[11pt]{article}
\usepackage{graphicx,amsmath,latexsym,amssymb,amsthm,amsfonts,amscd,graphics,
 float}

\input{xy}
 \xyoption{all}
\usepackage[colorlinks=black,urlcolor=black]{hyperref}

\newtheorem{theorem}{Theorem}[section]
\newtheorem{proposition}[theorem]{Proposition}

\newtheorem{corollary}[theorem]{Corollary}
\newtheorem{lemma}[theorem]{Lemma}

\newtheorem{conjecture}[theorem]{Conjecture}

\newcommand{\bigfrac}[2]{\frac{\textstyle #1}{\textstyle #2}}

\newcommand{\Z}{{\mathbb Z}}
\newcommand{\F}{{\mathbb F}}

\numberwithin{equation}{section}

\title{Generators for Group Homology and a Vanishing Conjecture}
\author{Joshua Roberts}
\date{}

\begin{document}

\maketitle

% ---------------------NAME---------------------------------

% ABSTRACT----------------------------------------------------------------

\begin{abstract} Letting $G=F/R$ be a finitely-presented group, Hopf's formula expresses the second integral homology of $G$ in terms of $F$ and $R$. Expanding on previous work, we explain how to find generators of $H_2(G;\mathbb{F}_p)$. The context of the problem, which is related to a conjecture of Quillen, is presented, as well as example calculations.
\end{abstract}

% ---------------------Body---------------------------------

\section{Introduction}

Exploiting a classical theorem due to Hopf, we presented a series of algorithms in \cite{roberts:10} that give upper 
bounds on
group homology in homological dimensions one and two, provided coefficients are taken in a finite field.
In particular, examples confirmed the results in \cite{anton:08}, as well a new result, concerning the rank two special 
linear group over rings of number theoretic interest. This paper can be viewed as both a sequel and 
expansion of the results in \cite{roberts:10}.
%\noindent We extend these algorithms in Section \ref{alg2} to enable us to find generating sets for $H_2(-;k)$.

The initial motivation for constructing the algorithms was to gain insight into special cases of 
a conjecture originally given by Quillen in 1971, which we briefly discuss in Section \ref{quillen-subsec}. 
However, since the algorithms in \cite{roberts:10} depend only upon Hopf's formula for $H_2$, the usefulness of these 
algorithms extends to groups beyond the scope of Quillen's Conjecture.
Moreover, the algorithms are distinct from existing methods of calculating low dimensional group
homology in that they give an upper bound on the homology of any finitely-presented
group, though the upper bound is, at times, very large.

The main contribution of this paper is in Section \ref{gen-hom-groups} wherein we present a technique that expounds on the algorithms in \cite{roberts:10}
to find explicit generators of these homology groups. The technique relies heavily upon the above 
mentioned Hopf's formula for 
the second homology group of a finitely-presented group; the calculations are carried out with the computational
algebra program GAP \cite{gap4}.

As a byproduct of the calculations related to Quillen's Conjecture we are involved in a long
term project of preparing a database for low dimensional group homology of linear groups over number fields and their rings of integers. This work will be extended to other classes of finitely-presented groups of interest to computational group theory and algebraic topology. The first set of these calculations is found in Section \ref{calc}.

We note that when it is clear from the context, we occasionally omit explicitly writing the ground ring of 
linear groups as well as homology coefficients.

\section{A Vanishing Conjecture}\label{quillen-subsec}

One motivational problem for low dimensional group homology, which is related to algebraic K-theory, is the
study of homology for groups $GL_j(R)$, where $GL_j$ is a finite rank $j$ general linear group
 and $R$ is the ring of
integers in a number field. An approach to this problem is to consider the
diagonal matrices inside $GL_j$. Let $D_j$ denote the subgroup
formed by these matrices. Then the canonical inclusions $D_j\subset
GL_j$ for $j=0,1,...$ induce homomorphisms on group homology with
$k$-coefficients 
\begin{equation}
\rho: H_i(D_j(R);k) \to H_i(GL_j(R);k).
\end{equation}
 In \cite{quillen:71} Quillen conjectured:
\begin{conjecture} \label{quillen-conjecture} The homomorphism $\rho$, as 
given above, is an epimorphism for $R=\mathbb Z[\zeta_p, 1/p]$, $p$ a regular odd prime, $\zeta_p$ a primitive 
$p$th root of unity, $k=\mathbb F_p$ and any values of $i$ and $j$.
\end{conjecture}

Conjecture \ref{quillen-conjecture} has been proved in a few cases and disproved in infinitely many other cases. For
$R=\Z[1/2]$ it was proved by Mitchel in \cite{mitchell:92} for $j=2$ and by Henn in \cite{henn:99} for $j=3$.
Anton gave a proof for $R=\Z[1/3,\zeta_3]$ and $j=2$ in \cite{anton:99}.

Dwyer gave a disproof for the conjecture for $R=\Z[1/2]$ and $j=32$ in \cite{dwyer:98} which
Henn and Lannes improved to $j=14$ in \cite{henn:95}; this is an improvement in light
of Henn's result in \cite{henn:96} that states that if
Conjecture \ref{quillen-conjecture} is false for $j_0$ then
it is false for all $j \ge j_0$. Anton disproved the conjecture
for $R=\Z[1/3,\zeta_3]$ and $j \ge 27$ also in \cite{anton:99}.
The interested reader should consult \cite{knudson:00} for more details.

This conjecture was reformulated and, in a sense, corrected by Anton:
\begin{conjecture}\textup{\cite{anton:03}} \label{anton-conjecture} Given $p, k$ and $R$ as above,
\begin{equation}
H_2(GL_2(R);k) \cong H_2(D_1(R);k).
\end{equation}
\end{conjecture}

Anton's conjecture led to a proof of Conjecture \ref{quillen-conjecture} for $\Z[1/5, \zeta_5]$ and $i=j=2$.
For a survey on the current status of conjectures \ref{quillen-conjecture} and \ref{anton-conjecture} we cite \cite{anton:08}.

\subsection{Reduction via a Spectral Sequence}
Given a group extension $$ 1 \to N \to G \to Q \to 1$$ there is the Hochschild-Serre Spectral Sequence 
\cite[p. 341]{mccleary:01} with
\begin{equation}\label{hsss}
E^2_{p,q} \cong H_p(Q; H_q(N;k)) \Longrightarrow H_{p+q}(G;k),
\end{equation}
where we take coefficients in a field $k$ regarded as a trivial $G$-module. We use this spectral sequence to
reduce a special case of Quillen's conjecture to an exercise in linear algebra.

\begin{lemma}
Fix $R=SL_2(\Z[\zeta_p, 1/p])$ and field of coefficients $k=\F_p$, 
\begin{equation}
H_2(GL_2(R);k) \cong (H_2(SL_2(R);k)_{GL_1(R)}/Im(\tau) \oplus H_2(GL_1(R);k),
\end{equation}
where, for a group $G$ and a $G$-module $M$, $M_G$ is the group of co-invariants and $\tau$ is the 
transgression map $E^3_{3,0} \to E^3_{0,2}$.
\end{lemma}

\begin{proof}
We first note that $R$ is a Euclidean ring \cite{cohn:66}, which, by Lemma 7.2 \cite{anton:03} 
implies that $SL_2(R)$ is a perfect group. Thus, applying the spectral sequence \ref{hsss} to the extension
\begin{equation}
1 \to SL_2(R) \to GL_2(R) \to GL_1(R) \to 1,
\end{equation}
 we see that the entries $E^2_{p,1}$ are all $0$.
 Thus for $q<3$ the $E^3$ page is equal to the $E^2$ page.

 We also note that
\begin{equation}
GL_1(R) \cong D_1(R) \cong R^\times, \label{dimarg3}
\end{equation}
where $R^\times$ is the group of units of $R$.

\begin{figure}[h]
\centering
\includegraphics[scale=.85]{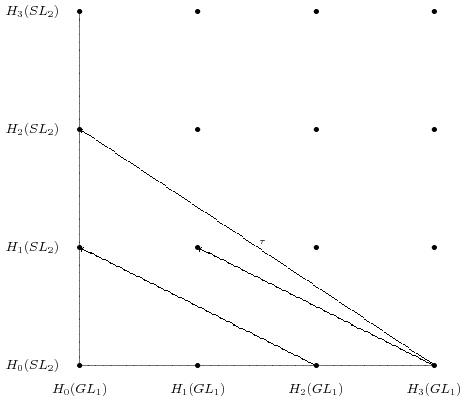}
\caption{$E^2$ page with $\tau:E^3_{3,0} \to E^3_{0,2}$ displayed \label{ss}}
\end{figure}

Figure \ref{ss} displays the $E^2$ page of this spectral sequence, and we have included the transgression $\tau: 
E^3_{3,0} \to E^3_{0,2}$ for reference.
Note that since
$E^2_{p,q} \cong E^3_{p,q}$ for all $p$ and for all $q<3$ then $E^2_{p,1} \cong E^\infty_{p,1}$. Moreover, 
$E^4_{p,q} \cong E^\infty_{p,q}$ for $p,q+1<4$ and
$E^4_{0,2} \cong H_2(SL_2(R))_{GL_1(R)} / Im (\tau)$. Since we have chosen field coefficients, 
any extension problems are trivial. Thus we have the following decomposition.
\begin{align}
H_2(GL_2(R)) &\cong E^4_{2,0} \oplus E^4_{1,1} \oplus E^4_{0,2}\\
& \cong  H_2(SL_2(R))_{GL_1(R)} / Im (\tau) \oplus H_2(GL_1(R)). \label{ss-equation}
\end{align}
\end{proof}

This immediately implies the following corollary.

\begin{corollary}
As vector spaces over $k$, 
\begin{equation}\label{dimarg}
\text{dim }H_2(GL_2(R);k) \ge \text{dim }H_2(GL_1(R);k)
\end{equation}
\end{corollary}

Recall from Section \ref{quillen-subsec} that the Quillen Conjecture implies that the map induced by inclusion
\begin{equation}
H_2(D_2(R))  \twoheadrightarrow H_2(GL_2(R)) \label{quillen}
\end{equation}
is surjective. Anton's reformulation of Quillen's conjecture in \cite{anton:08} and results in \cite{anton:03} imply that map \ref{quillen} factorizes thusly:
\begin{equation}
\xymatrix{
H_2(D_2) \ar@{->>}[rr] \ar[dr] &      & H_2(GL_2) \\
                         &   H_2(D_1) \ar@{->>}[ur] & }.
\end{equation}
Then $H_2(D_1(R)) \twoheadrightarrow H_2(GL_2(R))$ is surjective and
\begin{equation}
\text{dim }H_2(D_1) \ge \text{dim }H_2(GL_2) \label{dimarg2}.
\end{equation}
Then equations \ref{dimarg}, \ref{dimarg2}, and \ref{dimarg3} imply Conjecture \ref{anton-conjecture}
\begin{equation}
H_2(GL_2(R)) \cong H_2(GL_1(R)),
\end{equation}
which, by Equation \ref{ss-equation}, is equivalent to
\begin{equation}
H_2(SL_2(R))_{GL_1(R)}  \cong  Im (\tau),
\end{equation}
which is true if and only if $\tau$ is surjective. Moreover, for Conjecture \ref{anton-conjecture} to be true, it
is sufficient for the finitely-presented group $SL_2(\Z[1/p, \zeta_p])$ to have trivial second dimensional
$\F_p$-homology.

In this context, the purpose of \cite{roberts:10} was to give a series of algorithms that 
estimated the second homology group of any finitely-presented group.
More precisely, given a finitely-presented group $G$ and a
finite field $k$, the second homology group $H_2(G;k)$ with coefficients
in $k$ is a finite dimensional vector space over $k$. Our algorithms
give an upper bound for the dimension of $H_2(G;k)$ and, in
particular cases, the algorithms calculate precisely this dimension.

The algorithms confirmed results by Anton that Conjecture \ref{quillen-conjecture} holds for $R=SL_2(\Z[1/p, \zeta_p])$ 
and $k=\F_p$ for $p=3$ and $p=5$ (\cite{anton:03} and \cite{anton:08}).

\section{Generators of Homology Groups}\label{gen-hom-groups}

Let $1 \to K \stackrel{i}{\to} F \stackrel{q}{\to} G \to 1$ be an exact sequence of groups where $F$ is a finitely generated
free group and $K$ is finitely generated as an $F$-module with the $F$-action given by conjugation, $i$ and $q$ denote inclusion and
quotient homomorphism, respectively. That is, $G$ has finite presentation given  by the generators of $F$ modulo
then normal closure of $K$ in $F$.

\begin{theorem}[Hopf]Given $G,F,K$ as above, there is an exact sequence
$$ 1 \to [F,R] \to [F,F] \to H_2(G,\mathbb Z) \to 1.$$
\end{theorem}

This gives an exact sequence
$$
1 \to H_2(G,\mathbb Z) \to \bigfrac{R}{[F,R]} \to \bigfrac{F}{[F,F]} \to \bigfrac{F}{R[F,F]} \to 1. $$
The last two terms are finitely generated abelian groups and
algorithms exist to give their structure. Also in \cite{roberts:10}, we
explain how to use this exact sequence to find an upper bound
on the dimension of $H_2(G;k)$, where $k$ is the finite field of
prime characteristic $p$.

The inclusion homomorphism $i:K \to F$ induces a homomorphism
$i_*: A \to B$ where we have denoted $K/K^p[F,K]$ by $A$ and
$F/K^p[F,F]$ by $B$. Note that for $k \in K$ and $f \in F$ we have
that $[k,f]=k^fk^{-1}=1$ in $A$. Thus $k^f=k$ in $A$ which gives
that $A$ is a trivial $F$-module. Let $S_K$ be the set of
generators of $K$ as an $F$-module.

We note that the image of $i_*$ is generated by the set of all $i_*(k)$ for $k \in S_K$. Then since $B$ is a vector space over $k$,
there is a subset $S_K' \subset S_K$ such that $i_*(k')$ with $k \in S_K'$ is a basis for the image of $i_*$.

The primary interest is on the kernel of $i_*$, which is isomorphic to $H_2(G;k)$. As stated above, a previous paper gives an upper bound $n$
on the dimension of this vector space. We seek an explicit description of these $n$ elements of $S_K$. To this end, we restate two facts:

\begin{itemize}
\item $A$ is a vector space that is spanned by $S_K$
\item $S_K' \subset S_K$ is a subset with $i_*(S_K')$ a basis for the image of $i_*$ in $B$
\end{itemize}

Let $v \in A$, then $v= \displaystyle \sum_{\lambda \in S_K} c_\lambda \lambda$, where $c_\lambda \in k$, by $(1)$. Note that
$i_*(v)=0$ is equivalent to $\displaystyle \sum_{\lambda \in S_K} c_\lambda i_*(\lambda)=0$ in $B$.

Moreover, each $i_*(\lambda)= \displaystyle \sum_{\mu \in S_K'} a_{\lambda, \mu} i_*(\mu)$, where $a_{\lambda, \mu} \in k$, by $(2)$. Therefore,
$i_*(v)=0$ in $B$ if and only if
$$
\displaystyle \sum_{\mu \in S_K'} \left( \displaystyle \sum_{\lambda \in S_K} c_\lambda a_{\lambda, \mu} \right) i_*(\mu)=0 \text{ in }B,
$$
which is true if and only if
$$
\displaystyle \sum_{\lambda \in S_K} c_\lambda a_{\lambda, \mu}=0
$$
for all $\mu \in S_K'$. So we need to solve for the $c_\lambda$'s and find a basis for the solutions.

If $a \in A$ then $a=k_1^{f_1}k_2^{f_2} \cdots = k_1k_2 \cdots = k_1+k_2 + \cdots$ in $\mathbb F_p$. We want to use linear algebra in $B$ to find a basis
for the image of $i_*$ $\left\{ i_*(k) : k \in K_0 \right\}$.

$$
\xymatrix{& \bigfrac{F}{K^p[F,F]} \ar[dr] \\
\bigfrac{K}{K^p[F,K]} \ar[ur]^{i_*} \ar[rr]^{j_*} & &\bigfrac{F}{K^p[F,K]}}
$$ \\

\subsection{At the prime 7}
We consider the group $G=SL_2(\Z[1/7, \zeta_7])$, where $\zeta_p$ is a primitive $p$th root of unity. In
\cite{anton:08} it is proven that this group is generated by
$$
S=\{z,u_1,u_2,u_3,a,b,b_0,b_1,b_2,b_3,b_4,b_5,b_6,w\}
$$
modulo the relations
$$
\begin{array}{lcl}
R & = & \{b_t^{-1}z^{3t}bz^{3t}a,w^{-1}z^4u_1u_2u_3,z^7, [z,u_i], [u_i,u_j],a^4, [a^2,z], [a^2,u_i],\\
&& a^{-1}zaz, a^{-1}u_iau_i,\left[ b_s, b_t \right], b^{-3}a^2, b^{-3}b_0b_1b_2b_3b_4b_5b_6,b^{-7}_tw^{-1}b_t^{-1}w,\\
&& (b_0b_1^{-1}a^{-1}u_1)^3,(b_0b_2^{-1}a^{-1}u_2)^3,(b_0b_3^{-1}a^{-1}u_3)^3,\\
&& (b_0b_1^{-1}b_2^{-1}b_3a^{-1}u_1u_2)^3,  (b_0b_1^{-1}b_3^{-1}b_4a^{-1}u_1u_3)^3,(b_0b_2^{-1}b_3^{-1}b_5a^{-1}u_2u_3)^3,\\
&& (b_0b_1^{-1}b_2^{-1}b_3b_4b_5b_6^{-1}a^{-1}u_1u_2u_3)^3, a^{-2}b^{-1}u_ibz^{-3i}b^{-1}b_0^{-1}z^{3i}bz^{-i}u_i\}
\end{array}
$$
where $i,j \in \{1,2,3\}$ and $s,t \in \{1,2,3,4,5,6\}$.

That is, there is a short exact sequence $1 \to N(R) \to F(S) \to G \to 1$  with the set $S$ generating the free group
 $F(S)$ and the set $R$ normally generating the subgroup $N(R) \subset F(S)$.
 
We begin by reducing the number of generators and relators in $F(S)/N(R)$ in order to simplify the final calculations. Via GAP, 
it is easy to verify the following.
 
\begin{proposition}
There is an isomorphism of finitely-presented groups that maps the generators of the free group $F(S)$ to the
free group generated by $S'=\{ z,u_1,u_2,u_3,a,b_1 \}$ in the following way:
$$
\begin{array}{lcl}
 z & \mapsto & z\\
 u_1 & \mapsto & u_1 \\
 u_2 & \mapsto & u_2\\
 u_3 & \mapsto & u_3\\
 a & \mapsto & a\\
 b & \mapsto & z^{-3} b_1 z^3 a^{-1} \\
 b_0 & \mapsto & z^{-3} b_1 z^3\\
 b_1 & \mapsto & b_1\\
 b_2 & \mapsto & z^{3} b_1 z^{-3}\\
 b_3 & \mapsto & z^{-1} b_1 z\\
 b_4 & \mapsto & z^{2} b_1 z^{-2} \\
b_5 & \mapsto & z^{-2} b_1 z^2 \\
b_6 & \mapsto & z b_1 z^{-1}\\
w & \mapsto & z^{-2} u_1 z^{-1} u_2 u_3 .\\
\end{array}
$$

Moreover, the isomorphic finitely-presented group has set of 32 relations\\

\noindent $
R'=\{  zu_3z^{-1}u_3^{-1}, u_2u_3u_2^{-1}u_3^{-1},\\
 u_1u_2u_1^{-1}u_2^{-1}, \\
u_3au_3a^{-1}, \\
 u_1au_1a^{-1},\\
zu_2z^{-1}u_2^{-1},\\
 a^4, \\
u_1u_3u_1^{-1}u_3^{-1}, \\
 zaza^{-1}, \\
zu_1z^{-1}u_1^{-1}, \\
u_2au_2a^{-1}, \\
z^7,\\
b_1z^{-1}b_1zb_1^{-1}z^{-1}b_1^{-1}z, \\
b_1z^{-2}b_1z^2b_1^{-1}z^{-2}b_1^{-1}z^2,\\
z^{-3}b_1z^{-1}a^{-1}b_1z^{-1}a^{-1}b_1z^3a,\\
b_1z^{-3}b_1z^{-2}b_1^{-2}z^{-1}a^{-1}u_3a^{-1}z^{-1}b_1^{-1}u_3,\\
b_1z^{-1}b_1^{-2}z^{-1}b_1a^{-1}z^3u_2a^{-1}z^{-2}b_1^{-1}u_2,\\
b_1z^{-1}b_1z^{-3}b_1^{-2}z^{-1}u_1^{-1}a^{-2}z^{-2}b_1^{-1}u_1,\\  
b_1z^{-3}b_1z^3b_1^{-1}z^{-3}b_1^{-1}z^3, \\
z^{-1}b_1^-7zu_2^{-1}zu_3^{-1}u_1^{-1}zb_1^{-1}z^{-1}u_1z^{-1}u_2u_3, \\
b_1^-7u_2^{-1}zu_3^{-1}u_1^{-1}z^2b_1^{-1}z^{-3}u_1u_2u_3, \\
z^{-3}b_1^-7z^{-1}u_2^{-1}z^{-2}u_3^{-1}u_1^{-1}z^{-1}b_1^{-1}u_1u_2u_3, \\
\text{ \hspace{.5cm} }	zb_1^-7u_2^{-1}u_3^{-1}u_1^{-1}z^3b_1^{-1}z^{-3}u_1z^{-1}u_2u_3,\\
z^3b_1^-7u_2^{-1}z^{-2}u_3^{-1}u_1^{-1}z^{-2}b_1^{-1}zu_1u_2u_3,  \\
z^2b_1z^{-3}b_1z^{-3}b_1zb_1zb_1z^3b_1z^{-1}b_1a^{-2},\\
z^{-3}b_1z^{-1}b_1^{-1}u_2^{-1}a^{-1}b_1z^{-1}b_1^{-1}z^{-3}u_2^{-1}\\
\text{ \hspace{.5cm} }a^{-1}z^{-3}b_1z^{-1}b_1^{-1}z^{-3}u_2^{-1}a^{-1},\\
z^{-1}b_1^{-1}z^{-2}b_1z^3u_3^{-1}a^{-1}z^{-1}b_1^{-1}z^{-2}b_1z^3\\
\text{ \hspace{.5cm} }	u_3^{-1}a^{-1}z^{-1}b_1^{-1}z^{-2}b_1z^3u_3^{-1}a^{-1},\\
b_1^{-1}z^{-3}b_1za^{-1}z^{-2}u_1b_1^{-1}z^{-3}b_1z^3u_1^{-1}a^{-1}\\
\text{ \hspace{.5cm} }	b_1^{-1}z^{-3}b_1z^3u_1^{-1}a^{-1}, \\
b_1^{-1}z^{-1}b_1z^{-2}b_1z^{-1}b_1^{-1}a^{-1}z^3u_1u_2z^3b_1^{-1}\\
\text{ \hspace{.5cm} }	zb_1z^2b_1zb_1^{-1}u_1^{-1}a^{-1}u_2z^3b_1^{-1}zb_1z^2b_1zb_1^{-1}\\
\text{ \hspace{.5cm} }	u_1^{-1}a^{-1}u_2,\\  
b_1^{-1}z^{-1}b_1^{-1}z^{-2}b_1z^{-2}b_1z^{-1}u_1^{-1}z^{-1}a^{-1}u_3\\
\text{ \hspace{.5cm} }	z^{-3}b_1z^2b_1^{-1}zb_1^{-1}z^2b_1u_1^{-1}z^{-2}a^{-1}u_3z^{-1}b_1^{-1}z^{-2}\\
\text{ \hspace{.5cm} }	b_1z^3b_1^{-1}z^2b_1u_1^{-1}z^{-2}a^{-1}u_3,\\  
b_1^{-1}z^3b_1^{-1}z^{-1}b_1z^{-1}b_1za^{-1}\\
\text{ \hspace{.5cm} }z^{-2}u_3u_2z^{-3}b_1z^{-1}b_1^{-1}z^2b_1zb_1^{-1}z\\
\text{ \hspace{.5cm} }	u_3^{-1}a^{-1}u_2z^{-3}b_1z^{-1}b_1^{-1}z^2b_1zb_1^{-1}zu_3^{-1}a^{-1}u_2z^3, \\
zb_1z^{-3}b_1^{-1}z^{-3}b_1zu_1^{-1}za^{-1}u_2u_3z^{-3}b_1z^{-1}b_1^{-1}\\
\text{ \hspace{.5cm} }	z^{-1}b_1z^3b_1z^2b_1^{-1}zb_1^{-1}z^{-1}u_1^{-1}a^{-1}u_2u_3z^{-3}b_1z^{-1}\\
\text{ \hspace{.5cm} }	b_1^{-1}z^{-1}b_1z^3b_1z^2b_1^{-1}zb_1^{-1}z^{-1}u_1^{-1}a^{-1}u_2\\
\text{ \hspace{.5cm} }	u_3b_1^{-1}z^2b_1zb_1^{-1} \}.
$\\
\end{proposition}

By \cite{roberts:10}, the dimension of $H_2(G;\F_7)$ as a vector space over $\F_7$ is at most 6. We now seek generators of of this vector space. For simplicity, we denote $F(S')$ by $F$ and $N(R')$ by $N$. 
An application of the {{\sc{FindBasis}} algorithm from the same paper gives that 
$\bigfrac{N}{N^7[F,N]}$ is generated by the 12 elements
\begin{verbatim}
[ f1*f5*f1*f5^-1, 
f2*f3*f2^-1*f3^-1, 
f2*f5*f2*f5^-1,
f7*f5^-1*f7*f5^-1*f7*f5, 
f3*f7*f1^-2*f7*f1*f7^-2*f5^-1*f3*f1^-1*f5^-1*f7^-1,
f4*f7*f1^2*f7^-2*f1^2*f7*f5^-1*f1^-2*f4*f1^-1*f5^-1*f7^-1,
f2*f7*f1*f5^-1*f1^-2*f5*f7^-2*f1^3*f7*f5^-1*f2*f1^-1*f5^-1*f7^-1,
f1^2*f7^-7*f1^-1*f3^-1*f1^-1*f4^-1*f2^-1*f1^-2*f7^-1*f1^2*f2*f3*f4,
f7*f1*f7*f1^2*f7*f1*f7*f1^2*f7*f1^3*f7*f1^3*f7*f1*f5^-1*f1^-1*f5^-1,
f7*f1^2*f7^-1*f1^-1*f4^-1*f1^-1*f5^-1*f7*f1^2*f7^-1*f4^-1*f1^-2*
        f5^-1*f7*f1^2*f7^-1*f4^-1*f1^-2*f5^-1,
f1^-1*f7^-1*f1*f7*f1*f7*f1*f7^-1*f4^-1*f5^-1*f3*f1*f7^-1*f1*f7*f1^2*
        f7^-1*f1^-1*f7*f1^-1*f4^-1*f5^-1*f3*f1^-1*f7^-1*f1*f7*f1^2*f7^-1*
        f1^-1*f7*f1^-1*f4^-1*f5^-1*f3, 
f7*f1^-2*f7*f1*f7^-1*f1^2*f7*f1^2*f7^-1*f1*f7^-1*f1*f5^-1*f1^-2*f2*f3*
        f4*f7*f1^-2*f7*f1*f7^-1*f1^2*f7*f1^2*f7^-1*f1*f7^-1*f5^-1*f1^-3
        *f2*f3*f4*f7*f1^-2*f7*f1*f7^-1*f1^2*f7*f1^2*f7^-1*f1*
        f7^-1*f5^-1*f1^-3*f2*f3*f4 ]
\end{verbatim}

By reducing these elements in $\bigfrac{F}{[F,F]N^7}$ we obtain
\begin{verbatim}
[ <identity ...>, <identity ...>, <identity ...>, 
<identity ...>, <identity ...>, <identity ...>, 
f1^3, f7^-1, f2^-1, f1*f2, f1, f1^-1 ]
\end{verbatim}
Thus the last six elements form a basis for the 6-dimensional vector space over $\F_7$, 
$F/[F,F]N^7$. This implies that
the 6 vanishing elements are in the kernel of 
$$\bigfrac{N}{N^7[F,N]} \to \bigfrac{F}{N^7[F,F]}$$ 
and therefore are generators of $H_2$.  That is, the following is an explicit list of generators 
of $H_2\left( SL_2(\Z[1/7, \zeta_7]; \F_7 \right)$.

$$
\begin{array}{l}
\{ z a z a^{-1}, \\
u_1 u_2 u_1^{-1} u_2^{-1},\\
u_1 a u_1 a^{-1},\\
 b_1 a^{-1} b_1 a^{-1} b_1 a,\\
  u_2 b_1 z^{-2} b_1 z b_1^{-2} a^{-1} u_2 z^{-1} a^{-1} b_1^{-1}, \\
  u_3 b_1 z^2 b_1^{-2} z^2 b_1 a^{-1} z^{-2} u_3 z^{-1} a^{-1} b_1^{-1} \}
\end{array}
$$
However, the possibility still exists that any, or all, of these may be trivial in $H_2$.

\section{Homology Calculations}\label{calc}
% Needs to be rewritten to explain rewriting systems and confluence.

The following tables give the results of the algorithms in \cite{roberts:10} applied to various linear groups.
For the second table, a ``less than'' symbols indicates that the rewriting system involved in the calculation was not
 confluent, so only an upper bound was found. Otherwise, the rewriting system was confluent and the exact dimension 
was found; the code to implement these groups in GAP is given below. We note that while none of the 
results in the table below are new, the previous results were found by a wide variety of methods, many of which are not 
computational in nature.  

\begin{table}[h!]
\centering
$$
\begin{array}{|l|c|c|c|c|}
\hline
                    & H_1(-;\F_2) &  H_1(-;\F_3) & H_1(-;\F_5) & H_1(-;\F_7) \\\hline
GL_2(\Z)            &       2     &  0            &  0           & 0\\\hline
SL_2(\Z)            &     1       &   1           &  0           & 0\\\hline
SL_2(\Z_2)          &     1       &   0           &  0           & 0\\\hline
SL_2(\Z_3)          &    0        &  1            &  0           & 0\\\hline
SL_2(\Z_5)          &    0        &   0           &  1           &0 \\\hline
SL_2(\Z[i])         &      1      &     0         &   0          &0 \\\hline
SL_2(\Z[\omega]), \omega^3=-1   &    0        &   1           &   0          & 0\\\hline
SL_2(\Z[\sqrt{-5}]) &    3        &   2           & 1            &1 \\\hline
PSL_2(\Z)           &      1      & 1             & 0            & 0\\\hline
\end{array}
$$
\caption{Dimensions of First Homology Groups}
\end{table}
\begin{table}[h!]
\centering
$$
\begin{array}{|l|c|c|c|c|}
\hline
                   & H_2(-;\F_2) & H_2(-;\F_3) & H_2(-;\F_5) & H_2(-;\F_7) \\\hline
GL_2(\Z)           &    \le 4        & \le 2            &  \le 2         &   \le 2           \\\hline
SL_2(\Z)           &     \le 2       &   \le 2          &  \le 1         &   \le 1           \\\hline
SL_2(\Z_2)         &    1       &     0        &   0         &    0       \\\hline
SL_2(\Z_3)         &    0        &    1        &   0         &    0       \\\hline
SL_2(\Z_5)         &    0        &    0         &  1         &    0       \\\hline
SL_2(\Z[i])        &    1        &   0          &   0        &    0          \\\hline
SL_2(\Z[\omega]), \omega^3=-1   &     \le 1       &  \le 2           &    \le 1       &   \le 1          \\\hline
SL_2(\Z[\sqrt{-5}])&    \le 3        &    \le 3         &    0       &      0        \\\hline
PSL_2(\Z)          &    \le 1        &    \le 1         &    0       &      0        \\\hline
\end{array}
$$
\caption{Dimensions of Second Homology Groups}
\end{table}

\section{Conclusion}\label{conclusion}

The motivation for the algorithm used in the paper grew from work on Quillen's
conjecture. The utility of these algorithms is more general. In theory, they can 
be used to calculate or estimate the first and second homology of any 
finitely-presented group, provided homology coefficients are in a finite field.

Future work will involve refining and using the algorithms on a larger collection of sets with the goal of constructing the aforementioned database of calculations. In the context 
of the original problem, however, work to calculate the image of the transgression
$\tau$ in Figure \ref{ss} is necessary to make progress on the conjecture.

\bibliographystyle{abbrv}
\bibliography{ref}

\end{document}